\documentclass[12pt]{amsart}
\usepackage{amssymb,epsfig,amsmath,latexsym,amsthm}
\usepackage{graphicx}
\usepackage{epsfig,color}
\usepackage{enumitem}
\usepackage{comment}
\usepackage{tikz-cd}
\theoremstyle{plain}
\newtheorem{theorem}{Theorem}[section]
\newtheorem{lemma}[theorem]{Lemma}

\newtheorem{corollary}[theorem]{Corollary}

\newtheorem{conjecture}[theorem]{Conjecture}

\theoremstyle{definition}

\newtheorem{remark}[theorem]{Remark}

\usepackage{epsfig,color}

\makeatother

\theoremstyle{definition}

\def\fnum{equation}

\numberwithin{equation}{section}

\usepackage[hidelinks]{hyperref}

\begin{document}
\title
[Volume preserving stable]
{Isoperimetry and volume preserving stability  in real projective spaces}
\author{Celso Viana}

\address{\begin{flushright}
		Institute for Advanced Study  \\ School of Mathematics 
		\\
		1 Einstein Drive \\ Princeton NJ 08540 \\ USA
\end{flushright}}

\email{celsocsv@ias.edu
}

\address{\begin{flushright} Current Address of C. Viana\\
		Departament of Mathematics
		\\
		Universidade Federal de Minas Gerais \\ Belo Horizonte - MG 30123-970\\  Brazil
	\end{flushright}
} 

\email{celso@mat.ufmg.br}

\begin{abstract}
 We classify the  volume preserving stable  hypersurfaces in  the real   projective space $\mathbb{RP}^n$.   As a consequence, the solutions of the isoperimetric problem  are tubular neighborhoods of projective subspaces $\mathbb{RP}^k\subset \mathbb{RP}^n$ (starting with points). This  confirms a conjecture of Burago and  Zalgaller from 1988  and extends to higher dimensions previous result of M. Ritor\'{e} and A. Ros  on $\mathbb{RP}^3$.  We also derive an Willmore   type inequality for antipodal invariant hypersurfaces in  $\mathbb{S}^n$.
\end{abstract}

\maketitle

\section{Introduction}

\noindent The isoperimetric problem     looks for the   regions that minimize   perimeter among   compact sets  of a given fixed volume. The solutions are called isoperimetric regions and their boundaries isoperimetric hypersurfaces. The    isoperimetric inequality   beautifully describes    geodesic balls   as the only solutions  of the problem in the Euclidean space $\mathbb{R}^{n+1}$. The  classical proof   based on  symmetrization methods   works also in the  hyperbolic space $\mathbb{H}^{n+1}$ and in the round sphere $\mathbb{S}^{n+1}$.

Due to its variational nature, existence and regularity are relevant parts of the  problem. The   works of Almgren \cite{A} and  Schoen and Simon  \cite{SS} (see also Morgan \cite{M})  give a  satisfactory analytical description: when $M^{n+1}$ is closed or homogeneous,   isoperimetric hypersurfaces exist and are smooth except for a closed set of Hausdorff dimension $n-7$. The regular  part is a volume preserving stable  constant mean curvature hypersurface. In \cite{BC,BCE} Barbosa, do Carmo and Eschenburg  introduced   the notion of being volume preserving stable   in the broader class of  immersed  constant mean curvature hypersurfaces  and proved that only geodesic balls are stable  in the simply connected space forms.

In \cite{RR} Ritor\'{e} and Ros  studied volume preserving stable surfaces in  non-simply connected $3$-dimensional   space forms. Exploiting the Hopf quadratic differential of the surfaces, they proved that   volume preserving stable tori are flat. Moreover, using  the Hersch-Yau trick for constructing meromorphic maps, they also showed that sphere and tori are the only possible topologies for a  volume preserving stable surface in the real projective space $\mathbb{RP}^3$.  As a result,  the solutions of the isoperimetric problem  are   either  geodesic spheres or tubes about closed geodesics.  Inspired by this result, one can consider the projective spaces $\mathbb{P}^{n}(\mathbb{K})$, over the field $\mathbb{K}\in\{\mathbb{R}, \mathbb{C},\mathbb{H}\}$,  the most appealing spaces  one would like the isoperimetric problem solved. The following  conjecture, for $\mathbb{K}=\mathbb{R}$,  was posed   by Burago and Zalgaller  in 1988:
\begin{conjecture}[Berger \cite{B}, Burago and Zalgaller \cite{BZ}, and Ros \cite{Ro3}]\label{conjecture}
	The  isoperimetric regions  in the projective spaces $\mathbb{P}^{n+1}(\mathbb{K})$ are  tubular neighborhoods of projective subspaces $\mathbb{P}^k(\mathbb{K})\subset \mathbb{P}^{n+1}(\mathbb{K})$.
\end{conjecture}

Our main result address the real projective space of any dimension:

\begin{theorem}\label{main theorem}
If $\Sigma$ is a  compact two-sided volume preserving stable hypersurface  in $\mathbb{RP}^{n+1}$, then $\Sigma$ is  either a  geodesic sphere,   a quotient of a Clifford hypersurface $\mathbb{S}^{n_1}(R_1)\times \mathbb{S}^{n_2}(R_2)
\subset \mathbb{S}^{n+1}$, or  a two-fold covering of the  projective subspace $\mathbb{RP}^n$.
\end{theorem}

\begin{corollary}\label{corollary main result}
The isoperimetric regions in the real projective space $\mathbb{RP}^{n+1}$ are tubular neighborhoods of projective subspaces $\mathbb{RP}^k\subset \mathbb{RP}^{n+1}$.
\end{corollary}
  
\noindent  The corresponding result in $\mathbb{S}^{n+1}$ states that among  antipodal invariant regions  of a fixed volume, the least perimeter is
\[
\mathcal{A}= \{x\in \mathbb{S}^{n+1}: \sum_{i=1}^{r}x_i^2> a \}
\]
for some $r\in \{1, \ldots, n+1\}$ and $a\in \mathbb{R}_{+}$. For a discrete version of this result see \cite{EL}.

The following is a brief   account of  results on the isoperimetric problem. Rotationally symmetric surfaces were treated in  \cite{HHM1,R} (see also the survey \cite{HHM}).
The   spaces $\mathbb{H}^2\times\mathbb{R}$, $\mathbb{S}^2\times \mathbb{R}$, $\mathbb{H}^2\times \mathbb{S}^1$, and $\mathbb{S}^2\times \mathbb{S}^1$ were studied in \cite{HH,P,PR} via  symmetrization methods and ODE  stability analysis. 
The case $\mathbb{S}^1\times \mathbb{R}^n$ is also treated in \cite{PR}:  spheres and cylinders are the only  solutions when $n\leq 7$ (see also \cite{RR} when $n=2$).  The problem is also solved outside an explicit compact subset in the moduli space of flat  $T^2 \times \mathbb{R}$, see \cite{HPRR,RR2}. 
The $3$-dimensional spherical space forms with large fundamental group were studied by this author in \cite{V}. 
Finally, we mention the  works \cite{JM,Na} on the  description of small isoperimetric regions in general manifolds  and \cite{CESY,EM} on the uniqueness of large isoperimetric regions in  asymptotic flat manifolds with non-negative scalar curvature and positive mass.

The proof of Theorem \ref{main theorem} is inspired by the   work of do Carmo, Ritor\'{e} and Ros in \cite{dCRR} on the classification of two-sided index one  minimal hypersurfaces in the real projective space $\mathbb{RP}^{n}$. 
The   question of isoperimetry  in $\mathbb{RP}^n$ is  also raised in \cite{dCRR}.  
As in the volume preserving case,   index one minimal hypersurfaces   also minimize area for a certain class of variations. 
Despite   this similarity, the    stability analysis  for constant mean curvature brings several algebraic difficulties. Although there are many topologies for the index one minimal hypersurfaces, the norm of their second fundamental forms  can only  take   two possible values. This striking fact manifested sharply in the proof given in \cite{dCRR}. In contrast, the norm of the second fundamental form of  volume preserving stable hypersurfaces ranges over all  non-negative  numbers. 
A key observation that unite all  hypersurfaces  in Theorem \ref{main theorem} is that their second fundamental form   obey an equation of the form $A^2+\beta A -Id=0$ for some constant  $\beta$. This fact is reflected on the choice of vector fields  used in the second variation argument. Finally, we observed that the constant $\beta$  has a variational interpretation which we exploited to balance the vector fields. To handle possible singular  isoperimetric hypersurfaces in higher dimensions we used a cut-off  approach due to  Morgan and Ritor\'{e} \cite{MR} which  shows  that  singularities are negligible for the second variation argument.
The isoperimetric profile  is discussed in Section \ref{section proof}: we verified in lower dimensions that the  profile  is given by the perimeter of successive tubular neighborhoods of projective subspaces as conjectured by Berger \cite{B}.  

There is a natural comparison  between the isoperimetric profile and the Willmore energy of surfaces  as showed by Ros \cite{Ro2,Ro3}. This connection  was  investigated further by Marques and Neves    in their celebrated   proof of  the Willmore conjecture in \cite{MN}. Following these ideas, we add to the literature an Willmore   type inequality for antipodal invariant hypersurfaces in   $\mathbb{S}^n$ which generalizes the main result in \cite{Ro2}:

\begin{theorem}\label{willmore energy}
Let $\Sigma$ be a compact embedded  hypersurface separating $\mathbb{S}^{n+1}$ in two connected regions which are both antipodal invariant. Then 
\begin{eqnarray}\label{Willmore}
\int_{\Sigma} \bigg(1 + H^2\bigg)^{\frac{n}{2}} d_{\Sigma}\geq \sigma_n
\end{eqnarray}
with equality if, and only if, $\Sigma$ is congruent to the  Clifford hypersurface $C_{\lfloor\frac{n}{2}\rfloor, \lceil \frac{n}{2}\rceil}=\mathbb{S}^{\lfloor\frac{n}{2}\rfloor}\bigg(\sqrt{\frac{\lfloor\frac{n}{2}\rfloor}{n}}\bigg)\times \mathbb{S}^{\lceil\frac{n}{2}\rceil}\bigg(\sqrt{\frac{\lceil\frac{n}{2}\rceil}{n}}\bigg)$.
\end{theorem}

\begin{remark}
It is  a conjecture of  Solomon  that among  non-totally geodesic minimal hypersurfaces in $\mathbb{S}^{n+1}$, the Clifford hypersurface $C_{\lfloor\frac{n}{2}\rfloor, \lceil \frac{n}{2}\rceil}$ has the least area, see \cite{IW,FM,N}.  Theorem \ref{willmore energy}  confirms this conjecture in the  class of  minimal hypersurfaces with antipodal symmetry. 
In particular, $C_{\lfloor\frac{n}{2}\rfloor, \lceil \frac{n}{2}\rceil}$ has the least area among the  minimal   Clifford hypersurfaces $C_{p,n-p}=\mathbb{S}^p\bigg(\sqrt{\frac{p}{n}}\bigg)\times \mathbb{S}^{n-p}\bigg(\sqrt{\frac{n-p}{n}}\bigg)$ for every $n$. For partial results concerning Solomon's conjecture, see Ilmanen and White  \cite{IW}.
\end{remark}

\noindent Some related results: (i) Theorem \ref{main theorem} was obtained in \cite{ABP} under the additional  assumption that  the length of the second fundamental form   is constant, (ii) 
Ram\'{\i}rez-Luna \cite{RL}   proved     that the width of $\mathbb{RP}^n$ is realized by a   Clifford hypersurface.

\vspace{0.1cm}

\noindent\textit{Acknowledgements.} I would like to thank the support and hospitality of the Institute for Advanced Study  where this work was conducted and Fernando Cod\'{a} Marques for organizing the special year \textit{Variational Methods in Geometry}  and for his interest in this work. 

This material is based upon work supported by the National Science Foundation under grant No. DMS-1638352.

\section{Clifford hypersurfaces}

 The   Clifford hypersurface $T_r^{(n_1,n_2)}$    in $\mathbb{S}^{n+1}$  is defined as the product $\mathbb{S}^{n_1}(\cos r)\times \mathbb{S}^{n_2}(\sin r)$, where $n_1+n_2=n$ and $r$ is constant in $(0,\frac{\pi}{2})$. Each point $x$ in $T_r^{(n_1,n_2)}$ is written as  $x=(\cos(r)z,\sin(r)w)$, where $z\in \mathbb{S}^{n_1}$ and $w\in \mathbb{S}^{n_2}$. In these coordinates, the  unit normal vector field $N$ over   $T_r^{(n_1,n_2)}$ is given by   \[
 N(x)=\bigg(-\sin(r)\,z, \cos(r)\,w\bigg).
 \]
The second fundamental form   of $\mathbb{S}^{n_1}(\cos r)\times \mathbb{S}^{n_2}(\sin r)$  is given below:
\[
A=\begin{bmatrix}
-\frac{\sin(r)}{\cos(r)} I_{n_1}      & 0 \\
0      &  \frac{\cos(r)}{\sin(r)}I_{n_2}
\end{bmatrix}.
\]
Here, $I_{n_1}$ and $I_{n_2}$ denotes the identity maps in $\mathbb{R}^{n_1}$ and $\mathbb{R}^{n_2}$ respectively. The principal curvatures of $T_r^{(n_1,n_2)}$ are $- \frac{\sin(r)}{\cos(r)}$ with multiplicity $n_1$ and $\frac{\cos(r)}{\sin(r)}$ with multiplicity $n_2$. 
The family  $\{T_r^{(n_1,n_2)}\}_{ r \in (0,\frac{\pi}{2})}$ foliates $\mathbb{S}^{n+1}$ by constant mean curvature hypersurfaces starting at the $n_1$-dimensional sphere $T_0^{(n_1,n_2)}=\mathbb{S}^{n_1}\times \{0\}$ and ending at the $n_2$-dimensional sphere $T_{\frac{\pi}{2}}^{(n_1,n_2)}=\{0\}\times \mathbb{S}^{n_2}$.
Each Clifford hypersurface is antipodal symmetric and, hence, descends naturally to the real projective space $\mathbb{RP}^{n+1}$. The Jacobi operator $L= \Delta + n+ |A|^2$  on $T_r^{(n_1,n_2)}$ has the following  expression  \[L= \Delta + n_1+n_2 +n_1 \frac{\sin^2(r)}{\cos^2(r)} + n_2 \frac{\cos^2(r)}{\sin^2(r)}.\] 
We want to study the stability of $T_r^{(n_1,n_2)}$ when projected in $\mathbb{RP}^{n+1}$. 
The eigenvalues of the Laplacian $\Delta$ on $\mathbb{S}^{n_1}(\cos(r))\times \mathbb{S}^{n_2}(\sin(r))$ are 
\[
\frac{k_1(k_1+ n_1-1)}{\cos^2(r)}\quad +\quad \frac{k_2(k_2+n_2-1)}{\sin^2(r)},
\]
where $k_1$ and $k_2$ are non-negative integers.  The eigenfunctions are restrictions to $T_r^{(n_1,n_2)}$ of  products of  homogeneous harmonic polynomials of degree $k_1$ in $\mathbb{R}^{n_1+1}$ with homogeneous harmonic polynomials of degree $k_2$  in $\mathbb{R}^{n_2+1}$. These eigenfunctions  are  invariant by the antipodal map only when $k_1+k_2$ is even.  The  possible values of $k_1$ and $k_2$ to obtain the   smallest positive eigenvalue $\lambda$ are either $(k_1,k_2)=(1,1)$ or $(k_1,k_2)=(2,0)$ or $(0,2)$. 
The stability of $T_r^{(n_1,n_2)}$ is then equivalent to  $\lambda - n - |A|^2\geq 0$. This holds whenever
\[
 \sqrt{\frac{n_2}{n_1+2}}\leq \tan(r)\leq \sqrt{\frac{n_2+2}{n_1}}.
\] 
\begin{remark}\label{characterization clifford torus}
To each Clifford hypersurface $\mathbb{S}^{n_1}(\cos(r))\times \mathbb{S}^{n_2}(\sin(r))$, there  exists a number $\beta \in \mathbb{R}$, which depends on $n_1$, $n_2$, and $r$, such that the the second fundamental form $A$ satisfies:
	\begin{eqnarray}\label{eq Clifford torus}
	A^2 + \beta\,A - Id \equiv 0.
	\end{eqnarray}
Indeed, if $\beta$ solves the equation (\ref{eq Clifford torus}), then $\frac{\sin^2(r)}{\cos^2(r)}- \beta \frac{\sin(r)}{\cos(r)}-1=0$ and $\frac{\cos^2(r)}{\sin^2(r)}+ \beta \frac{\cos(r)}{\sin(r)}-1=0$. Since these equations are equivalent, the claim follows. Taking the trace in both sides of (\ref{eq Clifford torus}) gives $|A|^2-n +\beta\,nH=0$.
\end{remark}

\section{Results}
\subsection{Volume preserving stable}
A   two-sided  isometric immersion $\phi: \Sigma^n \rightarrow M^{n+1}$  has \textit{constant mean curvature} if, and only if, $\phi$ is a critical point of the area functional for  volume preserving variations, see  \cite[Section 2]{BCE}. The mean curvature $H$ is defined by $tr(A)= n\,H$, where $A$ is the second fundamental form of $\Sigma$. Recall that  $A: T_x\Sigma \rightarrow T_x\Sigma$ is defined by $A= -\overline{\nabla}N$, where $N$ is a unit normal vector over $\Sigma$. 
 The critical point $\phi$ is called \textit{volume preserving stable} if the second derivative of the area  is non-negative for such variations. 
 \vspace{0.2cm}
 
\noindent Equivalently, $\phi$  is volume preserving stable if  for every $f \in C^{\infty}(\Sigma)$ with compact support such that $\int_{\Sigma} f\,d_{\Sigma}=0$,   we have
\begin{equation}\label{stability inequality}
I(f,f)=-\int_{\Sigma}fL\,f\,d_{\Sigma} :=\int_{\Sigma} |\nabla f|^2 - (Ric_M(N,N)+ |A|^2)\,f^2\, d_{\Sigma}\geq 0.
\end{equation}
$L$ is the Jacobi operator of $\Sigma$  defined as $L= \Delta_{\Sigma} + Ric_M(N,N) + |A|^2$.

\noindent Let $\phi: \Sigma^n \rightarrow \mathbb{S}^{n+1}$ be a two-sided constant mean curvature immersion. 
\begin{lemma}	
The position vector $x$ and the  unit normal vector $N$ along $\phi: \Sigma\rightarrow \mathbb{S}^{n+1}$ satisfy the following equations:
\begin{eqnarray}\label{fundamental equations}
\Delta_{\Sigma} x +n\,x\,=\,nH\,N  \quad \text{and}\quad \Delta_{\Sigma} N + |A|^2\,N\,=\,nH\,x.
\end{eqnarray}
\end{lemma}
\begin{proof}
Recall first that $\Delta_{\Sigma} x= n\overrightarrow{H_E}$ where $\overrightarrow{H_E}$ is the mean curvature vector of $\Sigma$ in $\mathbb{R}^{n+2}$. On the other hand,  $\overrightarrow{H_E}= - x + \overrightarrow{H}$, where $\overrightarrow{H}$ is the  mean curvature vector of $\Sigma$ in $\mathbb{S}^{n+1}$. If  $a$ is a constant vector in $\mathbb{R}^{n+2}$, then $\Delta_{\Sigma}\langle N,a\rangle=e_i\,e_i \langle N,a\rangle - \nabla_{e_i}e_i\,\langle N,a\rangle$, where $\{e_1,\ldots,e_n\}$ is an orthonormal basis of $\Sigma$. A direct computation gives:
\begin{eqnarray*}
	\Delta_{\Sigma}\langle N,a\rangle&=& - \langle \overline{\nabla}_{e_i} A(e_i),a\rangle + \langle x,a\rangle \langle A(e_i),e_i\rangle + \langle A(\nabla_{e_i}e_i),a\rangle \\
	&=& - \langle \nabla_{e_i}A(e_i),a\rangle - \langle A(e_i),A(e_i)\rangle \langle N,a\rangle + \langle A(e_i),e_i\rangle\,\langle x,a\rangle \\
	&& + \langle A(\nabla_{e_i}e_i),a\rangle = - |A|^2\langle N,a\rangle + nH\,\langle x,a\rangle  -\langle div_{\Sigma}A,a\rangle.
\end{eqnarray*}
The Codazzi equation implies that $div_{\Sigma}A=0$ since $H$ is constant. Therefore, $\Delta_{\Sigma} \langle N,a\rangle=- |A|^2\,\langle N,a\rangle +nH\,\langle x,a\rangle$.
\end{proof}

\begin{lemma}\label{lemma 1}
If  $a$ is a  constant vector in $\mathbb{R}^{n+2}$ and $\phi: \Sigma^n \rightarrow \mathbb{S}^{n+1}$ has constant mean curvature $H$, then
\begin{enumerate}
\item $L\langle x,a\rangle x= (|A|^2-n)\langle x,a\rangle x + nH\langle N,a\rangle x + n H\langle x,a\rangle N + X_1$;
\\
\item $L \langle x,a\rangle N= nH\langle N,a\rangle N + nH\langle x,a\rangle x + X_2$;
\\
\item $L \langle N,a\rangle x= nH\langle N,a\rangle N + nH\langle x,a\rangle x + X_3$;
\\
\item $L \langle N,a\rangle N= (-|A|^2+n)\langle N,a\rangle N+nH\langle N,a\rangle x + nH\langle x,a\rangle N + X_4$.
\\
\end{enumerate}
where $X_i:\Sigma \rightarrow \mathbb{R}^{n+2}$ is a  vector field tangent to $\Sigma$ for each $i=1,\ldots,4$.
\end{lemma}
\begin{proof}
The Jacobi operator evaluated in a vector valued function $V:\Sigma \rightarrow \mathbb{R}^{n+2}$ is understood to be computed on each coordinate:
\[
L(V)= \sum_{i=1}^{n+2}L(\langle V, e_k\rangle)\,e_k,
\]
where $\{e_1,\ldots,e_{n+2}\}$ is the standard orthonormal basis in $\mathbb{R}^{n+2}$. Hence, for the vector fields in the lemma, each summand $\langle V,e_k\rangle$ is a product of functions   $u$ and $v$. The lemma  follows from combining the formula \[
\Delta (uv)= u\Delta v + v\Delta u + 2\langle \nabla u,\nabla v\rangle
\]
with the identities given in (\ref{fundamental equations}). The terms associated to the inner product of gradients corresponds to the tangent vector fields $X_i$ in the lemma. They are given below:
\[X_1= 2\, a^{\top}, \quad X_2=X_3= -2\,A(a^{\top}),\quad \text{and}\quad X_4=2\,A^2(a^{\top}),
\]
where $a^{\top}$ denotes the projection of $a\in \mathbb{R}^{n+2}$ onto $T_x\Sigma$.
\end{proof}
\begin{lemma}\label{jacobi function}
If the projection of $\Sigma$ is volume preserving stable in $\mathbb{RP}^{n+1}$, then  $f:\Sigma \rightarrow \mathbb{R}$ given by $f(x)=\langle x,a\rangle \langle N,b\rangle - \langle x,b\rangle \langle N,a\rangle$  satisfies $L(f)=0$ and 
\[
\int_{\Sigma}\bigg( \langle x,a\rangle \langle N,b\rangle - \langle x,b\rangle \langle N,a\rangle\bigg)d_{\Sigma}=0.
\]
\end{lemma}
\begin{proof}
	It follows from Lemma \ref{lemma 1} that $L(f)=0$. One can easily check that $I(f + t,f+t)<0$ for every constant $t\neq 0$. Hence, $\int_{\Sigma}f d_{\Sigma}=0$.
\end{proof}

		\noindent Given two vectors  $a,b \in \mathbb{R}^{n+2}$, we consider the vector valued function $\Phi_{a,b}: \Sigma \rightarrow \mathbb{R}^{n+2}$ defined by:
\begin{equation}\label{definition vector field}
\Phi_{a,b}=\,-\,\langle x,a\rangle x\, +\, \langle N,a\rangle N\,  +\, \langle x,b\rangle N. 
\end{equation}
\begin{lemma}\label{tangent vector field}
	The  Jacobi operator evaluated on $\Phi_{a,b}$ is given by:
	\begin{eqnarray*}
		L\Phi_{a,b}=-(|A|^2-n)\bigg(\langle N,a\rangle N +\langle x,a\rangle x\bigg) + nH\bigg(\langle x,b\rangle x + \langle N,b\rangle N\bigg) + X,
	\end{eqnarray*}
	where $X$ is the tangent vector field over $\Sigma$ given by:
	\begin{equation*}\label{tangent vector field}
	X=  2\bigg(A^2(a^{\top})- A(b^{\top})\,-\,a^{\top}\bigg).
	\end{equation*}
\end{lemma}

\noindent Next lemma generalizes an assertion for minimal hypersurfaces in \cite{dCRR}.

\begin{lemma}\label{isomorphism}Let $R$ be the linear map $R: \mathbb{R}^{n+2}\rightarrow \mathbb{R}^{n+2}$  defined by	
\[
R(b)= \int_{\Sigma} \langle x,b\rangle\,N\,d_{\Sigma}.	\]
If  the projection of $\Sigma$ is volume preserving stable in $\mathbb{RP}^{n+1}$ and it is not totally geodesic, then $R$ is an isomorphism.\end{lemma}\begin{proof}If $R$ is not an isomorphism, then there exists a non-zero vector $b \in \mathbb{R}^{n+2}$ such that $R(b)=0$. This suggest using the vector field $V= \langle x,b\rangle N$ in the stability inequality. 
	It follows from Lemma \ref{lemma 1} that\[I(V,V)=-\int_{\Sigma}\langle V,LV\rangle\,d_{\Sigma}= -nH \int_{\Sigma} \langle x,b\rangle \langle N,b\rangle\,d_{\Sigma}=0.\]Since $\Sigma$ is stable,  $LV=c$, for some vector $c\in \mathbb{R}^{n+2}$.  It follows from  Lemma \ref{lemma 1} that \[nH(\langle x,b\rangle x + \langle N,b\rangle N ) - 2A(b^{\top})= c.\] 
In particular, $\langle x, c- nHb\rangle=0$ which implies $c=nHb$ since $\Sigma$ is not totally geodesic. Moreover, since $\Sigma$ is invariant by the antipodal map  and $-2A(b^{\top})= c^{\top}$, we obtain that  $2\langle N,b\rangle= nH \langle x,b\rangle$. If $H \neq 0$, then
	\[
	0=\int_{\Sigma} \langle x,b\rangle\langle N,b\rangle\,d_{\Sigma}= \frac{nH}{2} \int_{\Sigma}\langle x,b\rangle^2\,d_{\Sigma}.
	\]
	This is a contradiction since $\Sigma$ is not totally geodesic. On the other hand, if $H=0$, then $\langle N,b\rangle=0$. Hence, $\nabla^2(\langle x,b\rangle)=- \langle x,b\rangle Id$. By Obata's Theorem in \cite{O}, $\Sigma$ is congruent to a round sphere. The Gauss equation then implies that $\Sigma$ is totally geodesic, contradiction.
 \end{proof}

\section{Proof of Theorem \ref{main theorem}}\label{section proof}


\begin{proof}
	Let $\widetilde{\varphi}: \widetilde{\Sigma}\rightarrow \mathbb{RP}^{n+1}$ denote the stable  immersion of $\widetilde{\Sigma}$ in $\mathbb{RP}^{n+1}$. Using locally constant functions we conclude from the stability assumption that $\widetilde{\Sigma}$ is connected. If there exists a lift $\varphi: \widetilde{\Sigma}\rightarrow \mathbb{S}^{n+1}$, then it follows from \cite{BCE} that $\widetilde{\Sigma}$ is totally umbilical and hence, a geodesic sphere.
		If such lift does not exist, then there exist a orientable double covering $\Sigma\rightarrow \widetilde{\Sigma}$ and an isometric immersion $\varphi: \Sigma \rightarrow \mathbb{S}^{n+1}$ such that $\varphi\circ s=-\varphi$, where $s$ is the involution associated to the covering $\Sigma \rightarrow \widetilde{\Sigma}$. The two-sided assumption implies that $\Sigma$ is orientable and $N\circ s= -N$. The volume preserving stability of $\widetilde{\Sigma}$ implies that $I(u,u)\geq 0$ for every function $u:\Sigma\rightarrow \mathbb{R}$ which is even, i.e., $u\circ s=u$, and satisfies $\int_{\Sigma} u\,d_{\Sigma}=0$.  Note that the vector field $\Phi_{a,b}$ defined in (\ref{definition vector field}) satisfies $\Phi_{a,b}\circ s=\Phi_{a,b}$.

		\begin{lemma}\label{expression quadratic}
		\begin{eqnarray*}
			I(\Phi_{a,b},\Phi_{a,b})= \int_{\Sigma} (|A|^2-n)\bigg( -\langle x, a\rangle^2 + \langle N,a\rangle^2 \bigg)d_{\Sigma}\quad\quad\quad \quad\quad\quad \quad\quad \\
			-2nH \int_{\Sigma}\bigg(\langle N,a\rangle  \langle N,b\rangle - \langle x, a\rangle\langle x,b\rangle + \frac{1}{2}\langle N,b\rangle \langle x,b\rangle \bigg) \,d_{\Sigma}  .\end{eqnarray*}
		\end{lemma}
		\begin{proof}
		Recall from Lemma \ref{tangent vector field}  that  \[L\Phi_{a,b}= -(|A|^2-n)(\langle x,a\rangle x + \langle N,a\rangle N) +nH(\langle x,b\rangle x + \langle N,b\rangle N) + X,\] where $X$ is a tangent vector over $\Sigma$. A straightforward computation gives 
		\begin{eqnarray*}	-\int_{\Sigma} \langle \Phi_{a,b},L\Phi_{a,b}\rangle d_{\Sigma}= \int_{\Sigma} (|A|^2-n)\bigg(-\langle x,a\rangle^2 + \langle N,a\rangle^2 +\langle N,a\rangle \langle x,b\rangle\bigg)\\	- \int_{\Sigma} nH\bigg(-\langle x,b\rangle \langle x,a\rangle + \langle N,b\rangle \langle N,a\rangle +  \langle N,b\rangle \langle x,b\rangle	\bigg).\end{eqnarray*}   We have from the identities (\ref{fundamental equations}) that\begin{eqnarray*}	\int_{\Sigma}(|A|^2-n)\langle N,a\rangle\langle x,b\rangle d_{\Sigma}= \int_{\Sigma} \langle N,b\rangle\Delta \langle x,a\rangle- \langle x,b\rangle \Delta \langle N,a\rangle \nonumber\\	\quad -n\,H\int_{\Sigma}\bigg(\langle N,a\rangle  \langle N,b\rangle - \langle x, a\rangle\langle x,b\rangle \bigg) d_{\Sigma}.  \end{eqnarray*} By   Stoke's Theorem,  the first integral in the right hand side above is zero. Hence,
		\begin{eqnarray}\label{integration by parts 2}	\int_{\Sigma}(|A|^2-n)\langle N,a\rangle\langle x,b\rangle d_{\Sigma}&=& \\ -n\,H\int_{\Sigma}\bigg(-\langle x, a\rangle\langle x,b\rangle &+& \langle N,a\rangle  \langle N,b\rangle \bigg) d_{\Sigma}. \nonumber  \end{eqnarray}
		The lemma follows by substituting formula (\ref{integration by parts 2}) in the expression $-\int_{\Sigma}\langle \Phi_{a,b},L\Phi_{a,b}\rangle\,d_{\Sigma}$  above.
		\end{proof}

We address first the case where the mean curvature $H$  is positive.

\vspace{0.2cm}	

\noindent For every linear map $\varphi: \mathbb{R}^{n+2}\rightarrow \mathbb{R}^{n+2}$ we consider the symmetric quadratic form $Q_{\varphi}: \mathbb{R}^{n+2}\times \mathbb{R}^{n+2}\rightarrow \mathbb{R}$ given by	\[Q_{\varphi}(a,b)= - \int_{\Sigma} \bigg\langle \Phi_{a,\varphi(a)}, L\Phi_{b,\varphi(b)}\bigg\rangle\,d_{\Sigma}.\]

\noindent \noindent Let $\mathcal{M}$ be the set $\mathcal{M}= SO(n+2)\times\mathbb{S}^{n+1}$ and let $F:\mathcal{M}\times \mathbb{R}\rightarrow \mathbb{R}$ be the function defined by $F(\varphi,a,\beta)=Q_{\beta\cdot \varphi}(a,a)$.
\begin{lemma}\label{eq critical point}
If  $(\varphi_0,a_0,\beta_0)\in \mathcal{M}\times \mathbb{R}$  is a critical point of   $F$ and the mean curvature $H$  is positive, then
\[
\int_{\Sigma} \Phi_{a_0,\beta_0\varphi_0(a_0)}\,d_{\Sigma}=0.
\]
\end{lemma}
\begin{proof}
The tangent space $T_{\varphi_0}SO(n+2)$  is given by the linear space $\{\varphi_0\cdot K \in \mathcal{L}(\mathbb{R}^{n+2}):\, K^{\top}=-K\}$. Taking the derivative of $F(\varphi,a,\beta)$ with respect to $\varphi$ and recalling the expression of $Q_{\varphi}(a,a)$ in Lemma \ref{expression quadratic} we have
\begin{eqnarray*}
0= -\frac{1}{2nH} DF(\varphi_0, a_0,\beta_0)(\varphi_0\cdot K)&=& \\
 \int_{\Sigma} -\langle x, a_0\rangle \langle x,\beta_0\varphi_0\cdot K(a_0)\rangle &+& \langle N,a_0\rangle \langle N,\beta_0\varphi_0\cdot K(a_0)\rangle +\\  \int_{\Sigma} \langle x,\beta_0\varphi_0(a_0)\rangle \langle N,\beta_0\varphi_0\cdot K(a_0)\rangle &=&
\int_{\Sigma}\bigg\langle \Phi_{a_0,\beta_0\varphi_0(a_0)}, \beta_0\varphi_0\cdot K(a_0)\bigg\rangle.
\end{eqnarray*}
Since $\{\varphi_0\cdot K(a_0):\, K^{\top}=-K \}=\langle \varphi_0(a_0)\rangle^{\perp}$, we conclude that the vector   $\int_{\Sigma}\Phi_{a_0,\beta_0\varphi_0(a_0)}\,d_{\Sigma}$ is parallel to $\varphi_0(a_0)$ unless $\beta_0=0$. On the other hand, we also have
\begin{equation}\label{eq0}
0= \frac{\partial F}{\partial \beta}(\varphi_0,a_0,\beta_0)= -2nH \int_{\Sigma} \bigg \langle \Phi_{a_0,\beta_0\varphi(a_0)}, \varphi_0(a_0)\bigg\rangle d_{\Sigma}.
\end{equation}
Hence, if $\beta_0\neq 0$, then $\int_{\Sigma} \Phi_{a_0,\beta_0\varphi_0(a_0)}\,d_{\Sigma}=0$. If $\beta_0=0$, then  (\ref{eq0}) is also true when $\varphi_0$ is replaced by any orthogonal map $\varphi$. This implies   $\int_{\Sigma} \Phi_{a_0,0}\,d_{\Sigma}=0$ as well. This completes the proof of the lemma. 
\end{proof}
\vspace{0.2cm}

\noindent Our goal is to prove existence of a  critical point $(\varphi_0,a_0,\beta_0)$ for $F$ using the standard mountain pass variational theory. For this we consider the max-min number defined below:
\[
m_1= \sup_{\phi\in[\mathcal{M}]}\min_{(\varphi,a, \beta)\in \phi(\mathcal{M})}\,F(\varphi,a,\beta),
\] 
where $[\mathcal{M}]$ is the class of maps $\phi: \mathcal{M}\rightarrow \mathcal{M}\times \mathbb{R}$ that are homotopic to $\phi_0(a,\varphi)=(a,\varphi,0)$. First, let us show that $m_1\leq 0$. Indeed, each connected component of \[
Z=\{( Id, v,\beta )\in \mathcal{M}\times \mathbb{R}: v\,\text{ a first eigenvector of}\, Q_{\beta\cdot Id}\}\] is unbounded in the directions $\beta \rightarrow \pm \infty$. Hence, every $\phi\in [\mathcal{M}]$ must satisfy $\phi(\mathcal{M})\cap Z\neq  \emptyset$. The claim now follow since the first eigenvalue of $Q_{\beta \cdot Id}$ is non-positive  for every $\beta$. 
Indeed, if $\{e_1,\ldots, e_{n+2}\}$ is an orthonormal basis of $\mathbb{R}^{n+2}$, then 
\begin{eqnarray*}
\sum_{i=1}^{n+2}Q_{\beta\,Id}(e_i,e_i)= \sum_{i=1}^{n+2}\int_{\Sigma} (|A|^2-n)\bigg( -\langle x, e_i\rangle^2 + \langle N,e_i\rangle^2 \bigg)d_{\Sigma}\quad\quad\quad \quad\\
	-2nH\beta \sum_{i=1}^{n+2} \int_{\Sigma}\bigg(\langle N,e_i\rangle  \langle N,e_i\rangle - \langle x, e_i\rangle\langle x,e_i\rangle + \frac{\beta}{2}\langle N,e_i\rangle \langle x,e_i\rangle \bigg) \,d_{\Sigma}=0.\end{eqnarray*}
\noindent  In particular, $m_1$ is well defined. Note that if $m_1$ is a  critical value of $F$, then it  will follow from Lemma \ref{eq critical point} and  the volume preserving stability of $\Sigma$ that $m_1\geq 0$, and hence $m_1=0$.  Next we study the behavior of $F$ at the ends of $\mathcal{M}\times\mathbb{R}$: 
\begin{lemma}\label{infinite}
	There exists  constants $C_1>0$ and $C_2>0$ such that $|\nabla F|(\varphi,a,\beta)\geq C_1$ for every  $|\beta|\geq C_2$.
\end{lemma}
\begin{proof}
Assume $(\varphi_l,a_l,\beta_l)$  is a sequence of points in $\mathcal{M}\times \mathbb{R}$ satisfying  $\lim_{l\rightarrow \infty}|\beta_l|=\infty$ and such that $|\nabla F|(\varphi_l,a_l,\beta_l)=\varepsilon_l\rightarrow
 0$. We have by Lemma \ref{eq critical point} that 
 \[
\bigg|\int_{\Sigma} \Phi_{a_l,\beta_l\varphi_l(a_l)}\bigg|= \bigg|\int_{\Sigma} -\langle x,a_l\rangle x + \langle N,a_l\rangle N+ \beta_l \langle x,\varphi_l(a_l)\rangle N\bigg|\leq C\,( \frac{\varepsilon_l}{\beta_l} +\varepsilon_l).
 \]
Note by compactness of $\mathcal{M}$ that $(\varphi_l,a_l)$ has a convergent subsequence. Since $|\beta_l|\rightarrow \infty$, we conclude that \[\lim_{l\rightarrow \infty}R(\varphi_l(a_l))= \lim_{l\rightarrow \infty}\int_{\Sigma} \langle x, \varphi_l(a_l)\rangle N\,d_{\Sigma}= 0.\] This  is a contradiction since $\varphi_l(a_l)\in \mathbb{S}^{n+1}$ and the linear map $R$ is an  isomorphism by Lemma \ref{isomorphism}. 
\end{proof}

\noindent By Lemma \ref{infinite}, the functional $F$ satisfies the Palais-Smale condition  (P.-S.): namely, any sequence satisfying $|F(u_i)|\leq c$ and $|\nabla F(u_i)|\rightarrow 0$ has a convergent subsequence. The following finite dimensional min-max principle is proved in Struwe \cite{S}: 

\noindent \textbf{Min-max principle.} \textit{Suppose $M$ is a complete Riemannian manifold and $F\in C^{\infty}(M)$ satisfies  (P.-S.). Also suppose that $\mathcal{E}$ is a collection of sets which is invariant with respect to any smooth semi-flow $\Psi: M\times [0, \infty)\rightarrow M$ such that $\Psi(\cdot, 0)=id$, $\Psi(\cdot, t)$ a diffeomorphism of $M$ for any $t\geq 0$, and $F(\Psi(u,t))$ is non-decreasing in $t$ for any $u\in M$. If 
\[
\mu:= \inf_{E\in \mathcal{E}}\sup_{u\in E}F(u)
\]
is finite, then $\mu$ is a critical value of $F$.}
\vspace{0.2cm}

\noindent Let us now address the case $H=0$ proved in \cite{dCRR}. The proof is included  for completeness.  It follows from Lemma \ref{isomorphism} the existence of an unique linear map $\varphi_1: \mathbb{R}^{n+2}\rightarrow \mathbb{R}^{n+2}$ such that \[\int_{\Sigma} \langle x,\varphi_1(a)\rangle N \,d_{\Sigma}=\int_{\Sigma} \bigg(\langle x,a\rangle x - \langle N,a\rangle N\bigg)d_{\Sigma}.	\]	In other words,  the vector field $\Phi_{a,\varphi_1(a)}$ defined in (\ref{definition vector field}) is balanced. By Lemma \ref{expression quadratic}, we have  	\[ 0\leq	I(\Phi_{a,\varphi_1(a)},\Phi_{a,\varphi_1(a)})= \int_{\Sigma} (|A|^2-n)\bigg( \langle N,a\rangle^2 - \langle x,a\rangle^2 \bigg)d_{\Sigma},	\]
for every $a\in \mathbb{R}^{n+2}$. On the other hand, if $\{e_1,\ldots,e_n\}$ is an orthonormal basis of $\mathbb{R}^{n+2}$ with $e_1=\frac{a}{|a|}$, then \[
\sum_{i=1}^{n+2}\int_{\Sigma} (|A|^2-n)\bigg( \langle N,e_i\rangle^2 - \langle x,e_i\rangle^2 \bigg)d_{\Sigma}=0.
\] Therefore, $I(\Phi_{a,\varphi_1(a)},\Phi_{a,\varphi_1(a)})=0$ for every vector $a\in \mathbb{R}^{n+2}$.

\vspace{0.2cm}

\noindent Despite the conclusion  obtained for $H>0$ be weaker than the one obtained for $H=0$, it is enough to complete the proof of the theorem:

\vspace{0.2cm}

\noindent \textbf{Assertion.} \textit{If  there are vectors $a_0\in\mathbb{R}^{n+2}$ and $b_0\in \mathbb{R}^{n+2}$ such that  $I(\Phi_{a_0,b_0},\Phi_{a_0,b_0})=0$ and  \[ \int_{\Sigma}\Phi_{a_0,b_0} \,d_{\Sigma}=0,\]
  then $\Sigma$ is congruent to either a geodesic sphere or a  Clifford hypersurface $\mathbb{S}^{n_1}(R_1)\times \mathbb{S}^{n_2}(R_2)\subset \mathbb{S}^{n+1}$.}
\vspace{0.2cm}

\noindent \textit{Proof.} It follows from the variational characterization of volume preserving stable hypersurfaces that $L\Phi_{a_0,b_0}=c$ for some  vector $c\in \mathbb{R}^{n+2}$. By Lemma \ref{tangent vector field}, this is equivalent to	\begin{eqnarray*}
	-(|A|^2-n)\bigg(\langle x,a_0\rangle x +\langle N,a_0\rangle N \bigg) &+& nH\bigg(\langle x,b_0\rangle x + \langle N,b_0\rangle N\bigg) \\
	&+& 2\bigg(A^2(a_0^{\top}) - A(b_0^{\top}) - a_0^{\top} \bigg) = c. 
\end{eqnarray*}   If we use the notation $r(x)=(|A|^2-n)(x)$, then 
\begin{equation}\label{eq6}
\langle x, c+r(x)a_0-nH\,b_0\rangle=0 \quad \text{and}\quad
\langle N, c+r(x)a_0- nH\,b_0\rangle=0,\end{equation} 
for every $x\in \Sigma$. In particular,  
\begin{equation}\label{eq7}
c^{\top}+ \langle x,a_0\rangle \nabla r +r(x) a_0^{\top} - nH\,b_0^{\top}=0.
\end{equation}
Let $x_0\in \Sigma$ be the point of maximum  for  the function $r(x)$. 
It follows from equations (\ref{eq6}) and (\ref{eq7}) that $c= nH\,b_0-r(x_0)\,a_0$. Similarly, if $x_1\in \Sigma$ is the point of minimum for  $r(x)$, then equations (\ref{eq6}) and (\ref{eq7})  implies that $c= nH\,b_0-r(x_1)\,a_0$. Hence,
\[
r(x_0)=r(x_1)
\]
 and $r(x)=(|A|^2-n)(x)$ is a constant function. Consequently,  we can find a constant $\beta_0$  for which $|A|^2-n -\beta_0nH=0$. Therefore,		\begin{eqnarray}\label{eq a}0=\int_{\Sigma} \bigg(|A|^2-n -\beta_0 \,nH\bigg)\langle N,a\rangle \langle x,a\rangle d_{\Sigma}= \quad\quad\quad\quad\quad\quad\quad\quad\quad\quad\quad\nonumber\\- nH\int_{\Sigma} \bigg(\langle N,a\rangle \langle N,a\rangle-\langle x,a\rangle \langle x,a\rangle  + \beta_0 \langle x,a\rangle \langle N,a\rangle \bigg)\end{eqnarray}
for every $a \in \mathbb{R}^{n+2}$. The second equality follows from the integration by parts   in (\ref{integration by parts 2}).  
If  $H=0$, then  $|A|^2\equiv n$ (Lemma \ref{jacobi function} and Lemma \ref{isomorphism}). By the result of   Chern, do Carmo, and Kobayashi \cite{CdCK} (see also Lawson \cite{L}), $\Sigma$ is  congruent to the Clifford hypersurface $\mathbb{S}^{p}(\sqrt{\frac{p}{n}})\times \mathbb{S}^{n-p}(\sqrt{\frac{n-p}{n}})$ for some $p$. From now on, $H>0$.
Since the quadratic in the right hand side above  is symmetric (Lemma \ref{jacobi function}), we obtain that	\[\int_{\Sigma} \bigg(-\langle x,a\rangle x+ \langle N,a\rangle N +\beta_0\langle x,a\rangle N \bigg) d_{\Sigma}=0,\] for every $a\in \mathbb{R}^{n+2}$. In other words, the vector fields $\Phi_{a,\beta_0 a}$ are balanced.   On the other hand,  we have by Lemma \ref{expression quadratic} that  \[I(\Phi_{a,\beta_0 a},\Phi_{a,\beta_0 a})=\int_{\Sigma}(|A|^2-n-\beta_0nH)\bigg(\langle N,a\rangle^2-\langle x,a\rangle^2+\beta_0\langle x,a\rangle\langle N,a\rangle
\bigg).
\]  Hence, $I(\Phi_{a,\beta_0 a},\Phi_{a,\beta_0 a})=0$ for every  $a\in \mathbb{R}^{n+2}$. It follows from the volume preserving stability of $\Sigma$  that  $L(\Phi_{a,\beta_0\cdot a})= c$. Since $\beta_0$ satisfies $|A|^2-n-\beta_0nH=0$, we conclude simultaneously that both  $c$ and the vector field $X$ from Lemma \ref{tangent vector field} are zero:  \[
	A^2(a^{\top}) - \beta_0\,A(a^{\top}) - \,a^{\top}=0,
	\]
for every $a\in \mathbb{R}^{n+2}$. Hence,  $A^2 - \beta_0\,A - Id= 0$ as a $2$-tensor. As a consequence, the principal curvatures of $\Sigma$ are constant and equal to:
	\[
	\mu_{\pm}=\frac{\beta_0\pm \sqrt{\beta_0^2 + 4}}{2}.
	\] 
	In other words, $\Sigma$ is an isoparametric hypersurface with at most two principal curvatures. If  the   principal curvatures are equal, then $\Sigma$ is  totally umbilical  and, hence, a geodesic sphere. If the  principal curvatures are  $\mu_{+}$ with multiplicity $n_1$ and $\mu_{-}$ with multiplicity $n_2$, then  $\Sigma$ is congruent to the Clifford hypersurface $\mathbb{S}^{n_1}(\cos(r))\times \mathbb{S}^{n_2}(\sin(r))$  for some $r\in (0,\frac{\pi}{2})$, see Theorem 3.29 in \cite{CR}.
\end{proof}

\subsection{The isoperimetric problem}
In this section we use Theorem \ref{main theorem} to discuss the isoperimetric problem in $\mathbb{RP}^n$. 

Let $M$ be a Riemannian manifold.   We denote the $(n+1)$-dimensional Hausdorff measure of  a region $\Omega\subset M$   by $|\Omega|$. The $n$-dimensional Hausdorff measure of  $\partial \Omega$ is denoted by $|\partial\Omega|$.
The class of  regions considered here are  those of finite perimeter, see \cite{BZ}.
A region $\Omega\subset M$ is called an \textit{isoperimetric region} if
\[|\partial \Omega|=\inf \{|\partial \Omega^{\prime}|\,:\,\Omega^{\prime} \subset M \quad \textit{and} \quad |\Omega^{\prime}|=|\Omega|\}.\]
In this case, the hypersurface $\Sigma=\partial \Omega$ is called an \textit{isoperimetric hypersurface}.
For a recent reference on the regularity of isoperimetric hypersurfaces see \cite{M}:

\begin{theorem}
	\textit{If $(M^{n+1},g)$ is a  closed or homogeneous Riemannian manifold, then for any $0<t< \text{vol}(M)$ there exists an isoperimetric region $\Omega$ satisfying $|\Omega|=t$. Moreover, $\Sigma=\partial \Omega$ is smooth up to a closed set of Hausdorff dimension $n-7$. The regular part is a constant mean curvature volume preserving stable hypersurface.}
\end{theorem}

\begin{lemma}[Morgan-Ritor\'{e} \cite{MR}]\label{cut-off}
Let $\Sigma^k$ be a smooth, embedded submanifold of bounded mean curvature in $\mathbb{R}^{n+1}$ with compact singular set $\Sigma_0= \overline{\Sigma}-\Sigma$, satisfying $\mathcal{H}^{k-2}(\Sigma_0)=0$. Then, given $\varepsilon>0$, there is a smooth function $\varphi_{\varepsilon}: \overline{\Sigma} \rightarrow [0, 1]$
supported in $\Sigma$ such that
\begin{enumerate}
	\item $\mathcal{H}^k(\{\varphi_{\varepsilon}\neq 1 \})< \varepsilon$;
	\item $\int_{\Sigma} |\nabla \varphi_{\varepsilon}|^2\,d_{\Sigma}<\varepsilon$;
	\item $\int_{\Sigma}|\Delta \varphi_{\varepsilon}|\,d_{\Sigma}<\varepsilon$.
\end{enumerate}
\end{lemma}

\noindent As pointed out in  \cite{MR},  Lemma \ref{cut-off}  extends the stability inequality  to the set of smooth  bounded functions $f:\Sigma \rightarrow \mathbb{R}$ with mean zero on $\Sigma$ and gradient in $L^2(\Sigma)$: 
\begin{eqnarray}\label{stability inequality 2}
I(f,f)=-\int_{\Sigma} f(\Delta f+ Ric(N)f+|A|^2\,f)\,d_{\Sigma}\geq 0.
\end{eqnarray}
Indeed, given such  $f$, we construct the function $f_{\varepsilon}= (\varphi_{\varepsilon}f)^+ - t_{\varepsilon}(\varphi_{\varepsilon}f)^-$ where $t_{\varepsilon}$ is a constant so that $f_{\varepsilon}$ has mean zero in $\Sigma$. Since support of $f_{\varepsilon}$ is in the regular part of $\Sigma$,  (\ref{stability inequality 2}) holds for $f_{\varepsilon}$. The claim now follows by taking $\varepsilon \rightarrow 0$ and applying Lemma \ref{cut-off}. 

The vector valued function $\Phi_{a,b}$ used in the proof of Theorem \ref{main theorem} is smooth and bounded on $\Sigma$; let us show that its gradient is in $L^2(\Sigma)$. By the Cauchy-Schwarz inequality, $|\nabla \langle \Phi_{a,b}, e_j\rangle|$ is  bounded from  above by linear combinations of the principal curvatures $k_i$ of $\Sigma$. Therefore, it suffices showing that $|A|\in L^2(\Sigma)$. As in \cite{MR}, let $u:\mathbb{RP}^{n+1}\rightarrow [0,1]$ be such that $u\equiv 1$ in a neighborhood $V$ of the singular set $\Sigma_0$, $\overline{\nabla}u$ bounded, and  $\int_{\Sigma} u\,d_{\Sigma}=0$. Since the stability inequality (\ref{stability inequality 2}) holds for $u$, we obtain  that
\[
\int_{V\cap \Sigma} (n+ |A|^2)\,d_{\Sigma} \leq \int_{\Sigma} |\nabla u|^2<\infty. 
\]
In particular, the index form (\ref{stability inequality 2}) is well defined on  each coordinate of $\Phi_{a,b}$. Moreover, the integration by parts (\ref{integration by parts 2}) follows directly from Lemma \ref{cut-off}. 

These remarks guarantee that the proof  of Theorem \ref{main theorem} extends to the singular setting except for a minor difference  at the rigidity analysis for $L\Phi_{a_0,b_0}=c$. Indeed, we are not allowed to pick the point of maximum for the function $r(x)=(|A|^2-n)(x)$ since it is necessarily infinite in the singular case. Nonetheless, the argument up to that point implies that $c=nH\, b_0-r(x_1)a_0$, where $x_1$ is the point of minimum. From (\ref{eq6}) we obtain
\[
\bigg(r(x)-r(x_1)\bigg)\langle x,a_0\rangle =0,
\]
for every $x\in \Sigma$. Therefore,  $r(x)=(|A|^2-n)(x)$ is constant and $\Sigma$  is a regular hypersurface.
\vspace{0.2cm}

\noindent The isoperimetric profile of $\mathbb{RP}^{n}$ is the function $I_{\mathbb{RP}^{n}}:[0,vol(\mathbb{RP}^{n})]\rightarrow \mathbb{R}$ defined by $I_{\mathbb{RP}^n}(v)= \inf\{|\partial \Omega|:\Omega \subset \mathbb{RP}^n\,\,and\,\, |\Omega|=v \}$. Conjecture \ref{conjecture} in the form below is due to Berger \cite{B}:
\begin{conjecture}[Berger \cite{B}]	The  isoperimetric profile of $\mathbb{P}^{n+1}(\mathbb{K})$ is  given by the perimeter of successive tubular neighborhoods of projective subspaces $\mathbb{P}^k(\mathbb{K})\subset \mathbb{P}^{n+1}(\mathbb{K})$.\end{conjecture}

\noindent Corollary \ref{corollary main result} confirms the conjecture in the real projective space $\mathbb{RP}^{n}$ except for the word successive. Numeric computations   verified this simple description of the profile   in the dimensions $n\leq 10$.  Figure \ref{RP7} highlights the word successive.

\begin{figure}[h]
	\begin{center}		\includegraphics[scale=0.20]{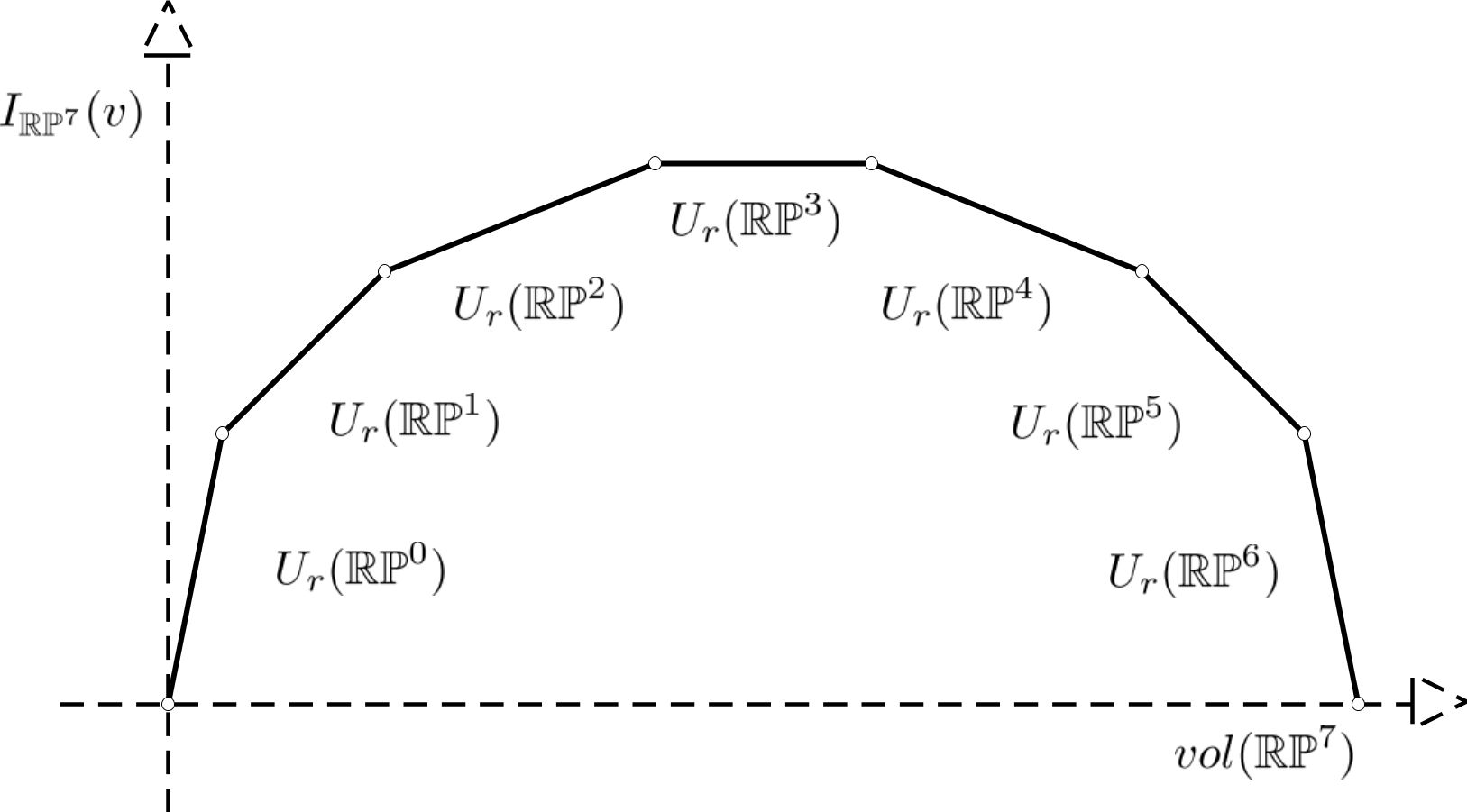}\\	\caption{Isoperimetric profile of $\mathbb{RP}^7$}\label{RP7}	
	\end{center}
\end{figure}

\section{Proof of Theorem \ref{willmore energy}}

\noindent  \textit{Arithmetic Mean - Geometric Mean inequality}: If $a_1,\ldots, a_{n}$ are positive numbers, then
\begin{eqnarray}\label{AGinequality}
\sqrt[n]{a_1\cdots a_n}\leq \bigg(\frac{a_1+ \ldots +a_n}{n}\bigg).
\end{eqnarray}
The equality  holds if, and only if, all the $a_i$'s are equal.

\begin{proof}[Proof of Theorem \ref{willmore energy}]
Consider the map $F_t: \Sigma \rightarrow \mathbb{S}^{n+1}$ given by $F(x,t)=\cos(t)x+ \sin(t)N(x)$. Given $p\in \Sigma$, take an orthonormal basis of $T_p\Sigma$ by principal directions. If $\{\lambda_1,\ldots,\lambda_{n}\}$ are the principal curvatures, then the Jacobian of the map $F_t$ is 
\[
Jac(F_t)= \prod_{i=1}^{n}\bigg(\cos(t)- \lambda_i\sin(t)\bigg).
\]
Assume for the moment that for each $i$  we have $\cos(t)-\lambda_i\sin(t)>0$. By the Arithmetic Mean - Geometric Mean Inequality (\ref{AGinequality}), we have:
\begin{eqnarray}\label{eq1}
Jac(F_t)\leq \bigg(\cos(t)-H\sin(t) \bigg)^{n}\leq \bigg(1+H^2\bigg)^{\frac{n}{2}}.
\end{eqnarray}
The last inequality follows from the Cauchy-Schwarz inequality applied to the vectors $\overrightarrow{u}=(\cos(t),\sin(t))$ and $\overrightarrow{v}=(1,-H)$. 

Let $U$ and $V$ be such that $\mathbb{S}^{n+1}\backslash\Sigma=U\cup V$. For any $t\in [-\pi,\pi]$, we consider the hypersurfaces $\Sigma_{|t|}=\{p\in U: dist(p, \Sigma) = |t|\}$ and $\Sigma_{-|t|}=\{p\in V: dist(p, \Sigma) = |t|\}$. The distance between $\Sigma$ and a point in $\Sigma_t$ is  attained by a minimizing geodesic which meets the hypersurface $\Sigma$ orthogonally. This distance is not bigger than the cut function, $c : \Sigma\rightarrow \{0 \leq t<\pi\}$. Recall that  $c(p)$ is the last positive time, such that the normal geodesic $t\mapsto \cos(t) p +\sin(t) N(p)$ attains the distance to $\Sigma$.
Thus, we have that $\Sigma_t=\{\cos(t)\,p + \sin(t)\, N(p): p \in \Sigma,\, c(p)\geq t\}$. We also have that $c(p)$ is less than or equal to the first focal value along the normal geodesic at $p$. Hence, for any  $0 \leq t \leq c(p)$, the Jacobian at $p$, $Jac(F_t)(p)$, of the map $F_t(p)=\cos(t)p + \sin(t) N(p)$ is nonnegative. Therefore,
\begin{eqnarray}\label{eq2}
\Sigma_t\subset F_t(\{p\in \Sigma: Jac(F_t)\geq 0 \}).
\end{eqnarray}
It follows from  (\ref{eq1}) and (\ref{eq2}) that
\begin{eqnarray}\label{eq4}
|\Sigma_t|\leq  \int_{\Sigma} \bigg( 1+H^2 \bigg)^{\frac{n}{2}}\,d_{\Sigma}.
\end{eqnarray}
 By assumption $U$ and $V$ are  antipodal invariant. Hence,  $U_t$ and $V_t$ defined by $\mathbb{S}^{n+1}\backslash \Sigma_t=U_t\cup V_t$ are also antipodal invariant. Moreover, the projection of  $\{\Sigma_t\}_{t\in [-\pi,\pi]}$ in $\mathbb{RP}^{n+1}$  provides an admissible class of sweep out to run the Almgren-Pitts Min-Max Theory. Hence,  there exists $t_0\in (-\pi, \pi)$ such that 
\begin{eqnarray}\label{eq5}
|\Sigma_{t_0}|\geq W(\mathbb{RP}^{n+1}).
\end{eqnarray}
Combining (\ref{eq4}), (\ref{eq5}), and the claim below, we obtain the  desired inequality (\ref{Willmore}).
   \begin{center}
	\textbf{Claim}. $W(\mathbb{RP}^{n+1})= \bigg| \mathbb{S}^{\lfloor\frac{n}{2}\rfloor}\bigg(\sqrt{\frac{\lfloor\frac{n}{2}\rfloor}{n}}\bigg)\times \mathbb{S}^{\lceil\frac{n}{2}\rceil}\bigg(\sqrt{\frac{\lceil\frac{n}{2}\rceil}{n}}\bigg)\bigg|$.
\end{center}
   If we have  equality in (\ref{Willmore}), then    (\ref{eq4}) and (\ref{eq1}) become equalities. The equality in (\ref{eq1})  occurs if, and only if, $\Sigma$ is totally umbilical or   both  $t_0=0$ and $H=0$. By assumption, $\Sigma$ is not totally umbilical and, hence,  $\Sigma$ is  minimal  and $t_0=0$. Since we also have equality in (\ref{eq5}), $\Sigma$ is congruent to $ \mathbb{S}^{\lfloor\frac{n}{2}\rfloor}\bigg(\sqrt{\frac{\lfloor\frac{n}{2}\rfloor}{n}}\bigg)\times \mathbb{S}^{\lceil\frac{n}{2}\rceil}\bigg(\sqrt{\frac{\lceil\frac{n}{2}\rceil}{n}}\bigg)$.

 \noindent  Proof of claim: by Zhou \cite{Z}, the width $W(\mathbb{RP}^{n+1})$ is realized by either a multiplicity one two sided index one  minimal hypersurface $\Sigma$  or a multiplicity two one sided      minimal hypersurface $\hat{\Sigma}$. As usual, $\Sigma$ and $\hat{\Sigma}$ are smooth away for a closed set of Hausdorff dimension $n-7$.
 By the result of do Carmo, Ritor\'{e}, and Ros \cite{dCRR}, the only possibility for the former is a Clifford hypersurface; the proof  still holds in the singular case by the cut-off argument of Morgan and Ritor\'{e}\cite{MR} discussed in the previous section. In  the latter case we  have  $|\hat{\Sigma}|\geq \mathbb |\mathbb{RP}^n|$. The claim now follows from the   inequalities:\begin{equation}\label{volume spheres}2 |\mathbb{S}^{n}|> \bigg|\mathbb{S}^p\bigg(\sqrt{\frac{p}{n}}\bigg)\times \mathbb{S}^{n-p}\bigg(\sqrt{\frac{n-p}{n}}\bigg)\bigg|\geq \bigg|\mathbb{S}^{\lfloor\frac{n}{2}\rfloor}\bigg(\sqrt{\frac{\lfloor\frac{n}{2}\rfloor}{n}}\bigg)\times \mathbb{S}^{\lceil\frac{n}{2}\rceil}\bigg(\sqrt{\frac{\lceil\frac{n}{2}\rceil}{n}}\bigg) \bigg|.\end{equation}
 To prove these inequalities, we argue as in Stone \cite{St}. We consider the function $f(p)=|\mathbb{S}^p||\mathbb{S}^{n-p}|(\frac{p}{n})^{\frac{p}{2}}(\frac{n-p}{n})^{\frac{n-p}{2}}$. Recall the standard formula for the volume of spheres in terms of the Gamma function $\Gamma$: \[|\mathbb{S}^d|= 2\pi^{\frac{d+1}{2}}\,/\, \Gamma\bigg(\frac{d+1}{2}\bigg).\] 
 Note that $f(x)=f(n-x)$ and  $\lim_{x\rightarrow 0}f(x)=2|\mathbb{S}^n|$. Our goal is to show that $f(x)$ is concave up on the interval $[0,n]$.  It is convenient to consider the function $\overline{f}(x)=\ln(f(x))$ instead. A computation gives
 \begin{eqnarray*}
 \overline{f}(x)=  2\ln(2) +\frac{n+2}{2}\ln(\pi) +\frac{x}{2}\ln(\frac{x}{n}) + \frac{n-x}{2}\ln(\frac{n-x}{n})\\
 -\ln\bigg(\Gamma(\frac{x+1}{2})\bigg)-\ln\bigg(\Gamma(\frac{n-x+1}{2})\bigg).
 \end{eqnarray*}
Using the identity $(\Gamma^{\prime}(t)/\Gamma(t))^{\prime}=\sum_{l=0}^{\infty}(t+l)^{-2}$, we compute:\begin{eqnarray*}	\overline{f}^{\prime\prime}&=& -\frac{1}{4}\bigg(\Gamma^{\prime}( \frac{x+1}{2})/\Gamma( \frac{x+1}{2})\bigg)^{\prime}- \frac{1}{4}\bigg(\Gamma^{\prime}(\frac{n-x+1}{2})/ \Gamma(\frac{n-x+1}{2})\bigg)^{\prime}\\
&& +\frac{1}{2x} +\frac{1}{2(n-x)} \\&=& - \frac{1}{4}\sum_{l=0}^{\infty}\frac{1}{(\frac{x+1}{2}+l)^2} - \frac{1}{4}\sum_{l=0}^{\infty}\frac{1}{(\frac{n-x+1}{2}+l)^2}+\frac{1}{2x} +\frac{1}{2(n-x)}.\end{eqnarray*}
As observed in \cite{St}, we have that $\sum_{l=0}^{\infty}\frac{1}{(x+\frac{1}{2}+l)^2}< \frac{1}{x}$. Applying this inequality to the expression above, we obtain $\overline{f}^{\prime\prime}(x)> 0$. This implies  $f^{\prime\prime}(x)>0$ and consequently   (\ref{volume spheres}).
\end{proof}

\end{document}